\documentclass[a4paper,11pt]{article}

\usepackage{amsmath, amssymb}
\usepackage[english]{babel}
  \usepackage{paralist}
  \usepackage{graphics} 
  \usepackage{epsfig} 
\usepackage{graphicx}
\usepackage{epstopdf}
 \usepackage[colorlinks=true]{hyperref}
 \usepackage{array}
\usepackage[latin1]{inputenc}
\usepackage{xcolor}
\usepackage{esint}
\usepackage{latexsym,amsfonts}
\usepackage{amsthm}

\usepackage{array}
\newtheorem{thm}{Theorem}
\newtheorem{theorem}{Theorem}[section]
\newtheorem{cor}{Corollary}[section]

\newtheorem{lem}[cor]{Lemma}

\newtheorem{prop}{Proposition}[section]

\theoremstyle{definition}

\newtheorem*{theorem*}{Theorem}

\newtheorem*{remark}{Remark}

\numberwithin{equation}{section}

\newcommand{\R}{\mathbb R}
\newcommand{\N}{\mathbb N}

\newcommand{\bigint}{\begin{picture}(10,10)
\put(-1,2){\line(1,0){10}}
\end{picture}\kern-14pt\int}

\begin{document}

\title{On the Dirichlet Problem for solutions of a restricted nonlinear mean value property }
\author{ANGEL ARROYO and  JOS\'{E} G. LLORENTE }
\date{Departament de Matem\`{a}tiques \\ Universitat Aut\`onoma de Barcelona \\ 08193 Bellaterra. Barcelona \\SPAIN \\
arroyo@mat.uab.cat \\jgllorente@mat.uab.cat }

\maketitle

\footnotetext{Partially supported by grants MTM2010-16232 and 2014
SGR 75.}

\begin{abstract}

We give an analytic proof of the solution of Dirichlet Problem for
continous functions satisfying a nonlinear mean value problem
related to  the $p$-laplace operator and certain stochastic games.

 \emph{Key
words:} mean value property, $p$-harmonious, dynamic programming
principle, Dirichlet problem. MSC2010: 31C05, 35B60, 31C45.
\end{abstract}


\section{Introduction}

Given a bounded domain $\Omega \subset \R^d$ and a function $r:
\Omega \to (0, +\infty )$, we say that $r$ is an \textbf{admissible
radius function} in $\Omega$ if
$$
0 < r(x) \leq dist(x, \partial \Omega )
$$
for all $x\in \Omega$. In this paper we will be mainly interested in
admissible radius functions satisfying  the lipschitz condition
\begin{equation}
|r(x) - r(y)| \leq |x-y| \, \, \, \, \, \, \text{for} \, \, x,y\in
\Omega
\end{equation}
When no confusion arises, we will also use the notation $r_x$
instead of $r(x)$ and $B_x$ instead of $B(x, r(x))$, the (open) ball
centered at $x$ of radius $r(x)$.

\

Suppose that $\mu$ is a fixed positive measure in $\R^d$. Given an
admissible radius function $r$ in $\Omega$  we define the operators
$S$ and $M$ in $C(\overline{\Omega})$, the space of continous
functions in $\overline{\Omega}$, by

\begin{align}
Su(x) = & \frac{1}{2} \big ( \sup_{B_x}u + \inf_{B_x}u \big ) \\
Mu(x) = & \fint_{B_x}\hspace{-0.2cm}u\, d\mu
\end{align}

Now let $0\leq \alpha < 1$. This paper is concerned with the
operator $T_{\alpha}$ obtained as a convex combination of $S$ and
$M$:

\begin{equation}
T_{\alpha} = \alpha S + (1 -\alpha )M
\end{equation}
in particular with the existence  and uniqueness of the Dirichlet
problem

\begin{equation}
\begin{cases} T_{\alpha} u  & = \,  u \, \, \, \,  \text{in} \, \, \, \Omega \\
u  & = \,  f \, \, \, \,  \text{on} \, \, \,
\partial \Omega
\end{cases}
\end{equation}
in the space $C(\overline{\Omega})$, for a given boundary data $f\in
C(\partial \Omega )$.

\

The motivation for this problem comes from the study of the so
called $p$-\textit{harmonious} functions. Let us first recall the
relation between the usual mean value property and harmonic
functions. It is well known that a continous function $u$ in a
domain $\Omega \subset \R^d$ is harmonic if and only if it satisfies
the \textit{mean value property}
\begin{equation}
u(x) = \fint_{B(x,r)} \hspace{-0.3cm} u \, dm
\end{equation}
for each $x\in \Omega$ and all $r$ with $0 < r < dist(x,
\partial \Omega )$. A classical theorem of Kellogg (\cite{K}) says
that  if  $\Omega$ is bounded, $u\in C(\overline{\Omega})$ and for
each $x\in \Omega$ there is an admissible radius function $ r(x)$
such that $(1.6)$ holds, then $u$ is harmonic in $\Omega$. In other
words, assuming continuity up to the boundary, the mean value
property for a single radius(depending on the point) implies
harmonicity or, with the notation introduced above, if $\mu = m$ is
Lebesgue measure in $\R^d$, $r: \Omega \to (0, +\infty ) $ is an
admissible radius function in $\Omega$ and $M$ is the operator given
by $(1.3)$ then $Mu = u$ implies that $u$ is harmonic. Observe that
this case corresponds to $\alpha = 0$ in $(1.4)$. Kellogg's theorem
is one of the most representative results in the so called
\textit{restricted mean value property} problems in classical
function theory (see the excellent survey \cite{NV} for these and
other similar questions).

\

The other extreme case, $\alpha = 1$, has been object of increasing
attention in the last years. If $S$ is the operator given by
$(1.2)$, associated to some admissible radius function in $\Omega$,
then functions satisfying $Su =u$ are called \textit{harmonious}
functions. The functional equation $Su = u$ appears in different
contexts, related to the problem of extending a continous function
on a closed subset to the whole space respecting its modulus of
continuity(\cite{LGA}), as a \textit{Dynamic Programming Principle}
in  tug-of-war games (\cite{PSSW}, \cite{R}), as a mean value
property related to the \textit{infinity laplacian} (\cite{MPR1},
\cite{L}) and also in connection with problems of image processing
(\cite{CMS}).

\

We briefly explain why nonlinear mean value properties are connected
to some distinguished nonlinear differential operators.

\

In the linear case, it follows from Taylor's formula that if $u\in
C^2 (\Omega )$,where $\Omega \subset \R^d$, then
$$
\lim_{r\to 0} \frac{1}{r^2} \Big ( \fint_{B(x,r)} \hspace{-0.2cm}u
\, dm - u(x) \Big ) = \frac{\triangle u (x)}{2(d +2)}
$$
for $x\in \Omega$. On the other hand,

$$
\lim_{r\to 0} \frac{1}{r^2} \Big [ \frac{1}{2}( \sup_{B(x,r)}u +
\inf_{B(x,r)}u ) - u(x) \Big ] = \frac{\triangle_{\infty}u
(x)}{2|\nabla u (x) |^2}
$$
where
$$
\triangle_{\infty}u = \sum_{i,j=1}^d u_{x_i}u_{x_j}u_{x_i , x_j}
$$
is the so called \textit{infinity laplacian} of $u$ (see \cite{ACJ},
\cite{L}). Another important differential operator is the
$p$-laplacian:
$$
\triangle_p u = div( \nabla u |\nabla u |^{p-2} )
$$
where $1 < p < \infty$. Direct computation shows that
$$
\triangle_p u = |\nabla u |^{p-2} \Big ( \triangle u + (p-2)
\frac{\triangle_{\infty}u}{|\nabla u |^2} \Big )
$$
which means that the $p$-laplacian can be interpreted as a sort of
average between the usual laplacian and the infinity laplacian. It
therefore makes sense to consider averages of the operators $M$ and
$S$ as in $(1.4)$. We may wonder whether the mean value property
\begin{equation}
u(x) = \frac{\alpha}{2}\Big (\sup_{B(x,r)} u + \inf_{B(x,r)} u \Big
) + (1- \alpha ) \fint_{B(x,r)} \hspace{-0.3cm}u \, dm
\end{equation}
is related to the $p$-laplacian for some specific value of $\alpha$,
possibly depending on $d$ and $p$. If $p\geq 2$, it turns out that
this is actually the case for the value
\begin{equation}
\alpha = \frac{p-2}{d+p}
\end{equation}
See \cite{MPR1}, \cite{MPR2} for the precise interpretation of this
relation.

\

The aim of  this paper is to study the existence and uniqueness of
the Dirichlet problem for the operator $T_{\alpha}$ under
appropriate assumptions on the domain $\Omega$ and the radius
function $r(x)$.

In the case of the operator $S$ it follows, as a particular case of
results in \cite{LGA}, that if $\Omega \subset \R^d $ is a bounded
and convex domain and $r$ is an admissible radius function in
$\Omega$ satisfying $(1.1)$ then the Dirichlet problem

\begin{equation}
\begin{cases} Su   & =  \, u \, \, \, \,  \text{in} \, \, \, \Omega \\
u  & =  \, f \, \, \, \,  \text{on} \, \, \,
\partial \Omega
\end{cases}
\end{equation}
has a unique solution for each $f\in C(\overline{\Omega})$, where
$S$ is the operator given by $(1.2)$, associated to the radius
function $r$.

\

On the other hand, if $r(x)= \epsilon $ is constant then (not
necessarily continous) functions satisfying $T_{\alpha} u = u$ for
the value of $\alpha$ given by $(1.8)$ have been called
$p$-harmonious functions in \cite{MPR2}. Let $f\in C(\partial \Omega
)$ and $p\geq 2$. Since the balls $B(x, \epsilon )$ will eventually
leave $\Omega$ if $x$ is close to $\partial \Omega$, the authors in
\cite{MPR2} extend $f$ continuously to the strip
$$
\{ x\in \R^d \setminus \Omega \, : \, dist(x,
\partial \Omega ) \leq \epsilon \}
$$
and they prove that there is a unique $p$-harmonious function
$u_{\epsilon}$ having $f$ as boundary values (in this extended
sense). Furthermore, if $\Omega$ satisfies some regularity
assumptions, then  $u_{\epsilon} \to u$ uniformly in
$\overline{\Omega}$, where $u$ is the unique solution of the
Dirichlet problem
$$
\begin{cases}
\triangle_p u   & =  \, 0 \, \, \, \,  \text{in} \, \, \, \Omega \\
u  & =   \, f \, \, \, \,  \text{on} \, \, \,
\partial \Omega
\end{cases}
$$
(See \cite{MPR2}, Theorem $1.6$). In order to prove existence and
uniqueness of the $p$-harmonious Dirichlet problem, the authors in
\cite{MPR2} use the interpretation of the functional equation
$T_{\alpha}u = u$ as a dynamic programming principle for a
tug-of-war game. See also \cite{LPS} for a purely analytic proof.

\

Motivated by the classical results on the restricted mean value
property for the operator $M$ (as Kellogg's theorem) and the results
in \cite{LGA} for the operator $S$, we consider the operator
$T_{\alpha}$  in the setting of admissible radius functions in
domains instead of the constant radius case. Furthermore, this makes
possible to avoid extending the boundary function outside the
domain.


We  recall that a convex domain $\Omega \subset \R^d$ is
\textbf{strictly convex} if $\partial \Omega$ does not contain any
segment or, equivalently, if for any $x$, $y\in \partial \Omega$ the
open segment $(x,y)$ is contained in $\Omega$. For technical reasons
that will become apparent in section $2$ we also introduce the
operator
\begin{equation}
H_{\alpha} = \frac{1}{2}(I + T_{\alpha})
\end{equation}

We state now the main result of the paper.

\begin{thm}
Let $\Omega \subset \R^d$ be an strictly convex bounded domain in
$\R^d$, $0\leq \alpha < 1$, $0 < \epsilon < 1 - \alpha$  and
$\lambda , \beta $ such that
\begin{align}
1 \leq \beta <  & \frac{\log ( \frac{1}{\alpha})}{\log ( \frac{1}{1
- \epsilon})}\\
0 < \lambda < & \min \{ \epsilon , \frac{1}{\beta}\} \Big
(\frac{2}{diam (\Omega ) } \Big )^{\beta -1}
\end{align}
Suppose that  $r: \Omega \to (0, +\infty )$ is an admissible radius
function in $\Omega$ satisfying $(1.1)$ together with the following
conditions:
\begin{equation}
\lambda \, dist(x, \partial \Omega)^{\beta} \leq r(x) \leq \epsilon
\, dist(x, \partial \Omega )
\end{equation}
for all $x\in \Omega$. If $S$ and $M$ are the operators given by
$(1.2)$ and $(1.3)$, associated to the admissible radius function
$r(x)$ and to $d$-dimensional Lebesgue measure $\mu = m$ and $
\displaystyle T_{\alpha} = \alpha S + (1-\alpha) M $ then for any
$f\in C(\partial \Omega )$ the Dirichlet problem
\begin{equation}
\begin{cases}
T_{\alpha} u   & =  \, u \, \, \, \,  \text{in} \, \, \, \Omega \\
u  & =   \, f \, \, \, \,  \text{on} \, \, \,
\partial \Omega
\end{cases}
\end{equation}
has a unique solution $ \widetilde{u}\in C(\overline{\Omega})$.
Furthermore, if $u_0 \in C(\overline{ \Omega } )$ is any continuous
extension of $f$ to $\overline{\Omega}$ then $\displaystyle
H^k_{\alpha}u_0 \to \widetilde{u}$ as $k\to \infty$, uniformly in
$\overline{\Omega}$, where $H_{\alpha}$ is the operator given by
$(1.10)$.
\end{thm}

\begin{remark}
If $\Omega \subset \R^d$ is bounded and convex and $\beta$,$\lambda$
are as in $(1.11)$ and $(1.12)$ then $\displaystyle r(x) = \lambda
\, dist(x, \partial \Omega )^{\beta} $  is an admissible radius
function in $\Omega$ satisfying the requirements in Theorem $1$.
Note that the fact that $r$ satisfies the lipschitz condition
$(1.1)$ is a consequence of $(1.12)$ and Proposition $2.2$.
\end{remark}

The key ingredients in the proof of Theorem $1$ are to show  that
the sequence of iterates $\displaystyle \{ T^k_{\alpha}u_0 \}$ is
equicontinuous in $\overline{\Omega}$ together with a  regularity
result in metric fixed point theory due to Ishikawa(Theorem $4.1$).
It should be pointed out that the arguments necessary to prove the
interior equicontinuity (section $2$) and the boundary
equicontinuity(section $3$) are different in nature. Section $2$
works for general doubling measures in $\R^d$ but requires certain
rigid assumptions on the radius function. As for section $3$, we
have adapted a clever argument in \cite{J} for the usual (linear)
mean value operator $M$(when $\mu = m$ is Lebesgue measure) to the
nonlinear operator $T_{\alpha}$. It turns out that the trick in
\cite{J} of using strict convexity to get equicontinuity of the
sequence of iterates at the boundary also works in our (nonlinear)
situation. The arguments in section $3$ work for general admissible
radius functions but they require $\mu$ to be Lebesgue measure.
However, for tentative future developments of the theory, we have
preferred to state section $2$ in the general doubling measure case.
Note that in our setting we cannot assume the a priori existence of
the solution as it happens for the usual mean value property.

\section{Interior estimates and interior equicontinuity.}

We start the section with an auxiliary result about doubling
measures.

\begin{lem}
Let $\mu$ be a doubling measure on $\R^d$. There are constants $C
>0$ and $0 < \theta \leq 1$ depending only on $\mu$ and $d$ such that, if  $a\in \R^d$ and
$0 < r \leq R$, then
\begin{equation}
\frac{\mu \big( B(a,R) \setminus B(a, r) \big ) }{\mu ( B(a, R) ) }
\leq C \Big ( \frac{R-r}{R}   \Big )^{\theta}
\end{equation}
\end{lem}

\begin{proof}
If $a\in \R^d$ is fixed and $0 <t< s $ , denote $A_{t,s} = B(a,s)
\setminus B(a,t)$. The lemma is a consequence of the following

\

\textbf{Claim}: there exists $0 < \delta < 1$ only depending on
$\mu$ and $d$ such that if $0 < t < s$ then

\begin{equation}
\mu  \big( A_{\frac{t+s}{2}, s} \big ) \leq \delta \mu \big (
A_{t,s}\big )
\end{equation}

\

To prove the claim, choose first a finite number of points $\{ \xi_k
\}_{k=1}^{N}$ on the unit sphere $S^{d-1}$ in such a way that the
spherical caps $\displaystyle C_k = B( \xi_k , \frac{s-t}{s}) \cap
S^{d-1} $ cover $S^{d-1}$ with finite overlapping(overlapping number
only depending on $d$).Now define the interior and exterior
spherical sectors as

\begin{align*}
Q^{in}_k = & a + \{ \rho \eta \, : \rho \in \big (t , \frac{t+s}{2}
\big ) ,
\eta \in C_k \} \\
Q^{ex}_k = & a + \{ \rho \eta \, : \rho \in \big (\frac{t+s}{2} , s
\big ) , \eta \in C_k \}
\end{align*}

By the doubling property, there is a  constant $D\geq 1$ such that ,
for each $k$, we have
$$
\mu (Q^{ex}_k ) \leq D \, \mu ( Q^{in}_k )
$$
Summing up over $k$ we get
$$
\mu \big ( A_{\frac{t+s}{2} , s} \big ) \leq D \, \mu \big ( A_{t,
\frac{t+s}{2}} \big )
$$
In particular
$$
\mu \big ( A_{\frac{t+s}{2}, s} \big ) \leq \frac{D}{D+1} \mu \big (
A_{t,s}\big )
$$
and $(2.2)$ follows by taking $ \delta = \frac{D}{D+1} $. Now, to
prove the lemma from the claim choose an integer $m$ such that
$$
(1 - 2^{-m})R < r \leq (1 - 2^{-(m+1)})R
$$
and apply the claim iteratively to the spherical shells
$\displaystyle A_{(1-2^{-k})R, R}$ for $0\leq k \leq m$ to obtain
\begin{equation}
\mu \big ( A_{r,R} \big ) \leq \mu \big ( A_{(1- 2^{-m})R, R}\big )
\leq \delta^m \mu (B(a,R))
\end{equation}
Then $(2.1)$ follows from $(2.3)$ by taking $\delta = 2^{-\theta}$
and $C = \delta^{-1}$.

\end{proof}


\begin{lem}
Let $\mu$ be a doubling measure on $\R^d$. There are constants $C>0$
and $0 < \theta \leq 1$ depending only on $\mu$ and $d$ such that if
$x$, $y\in \R^d$, $r_x>0$, $r_y>0$ and
\begin{equation}
|r_x - r_y | \leq |x-y| \leq \frac{r_x}{2}
\end{equation}
then
$$
\frac{\mu \big(B(x, r_x ) \setminus B(y, r_y )\big)}{\mu \big ( B(x,
r_x ) \big ) }  + \frac{\mu \big(B(y, r_y ) \setminus B(x,
r_x)\big)}{\mu \big ( B(y, r_y ) \big ) } \leq C \Big (
\frac{|x-y|}{r_x} \Big )^{\theta}
$$
\end{lem}

\begin{proof}
Observe that, from  $(2.4)$, $\displaystyle r_y \geq \frac{r_x}{2}$
and
\begin{align*}
B(x, r_x ) \setminus B(y, r_y ) & \subset B(x, r_x ) \setminus B(x,
r_y - |x-y| ) \\
B(y, r_y ) \setminus B(x, r_x ) & \subset B(y, r_y ) \setminus B(y,
r_x - |x-y| )
\end{align*}
 The conclusion follows from $(2.4)$ and Lemma $2.1$ with the
choices $R= r_x$, $r= r_y - |x-y|$ and $R=r_y$, $r= r_x - |x-y|$.
\end{proof}

\

The following lemma provides an interior estimate of the modulus of
continuity of $Mu$ where $M$ is the operator given by $(1.3)$.
\begin{lem}
Let $\mu$ be a doubling measure on $\R^d$, $\Omega \subset \R^d$ a
domain, $r$ an admissible radius function in $\Omega$ satisfying
$(1.1)$ and $u\in L^{\infty}(\Omega )$. There are constants $C>0$
and $0 < \theta \leq 1$, only depending on $\mu$ and $d$, such that,
if $x$, $y\in \Omega$, then
\begin{equation}
|Mu(x) - Mu(y)| \leq C ||u||_{\infty} \Big ( \frac{|x-y|}{r_x} \Big
)^{\theta}
\end{equation}
\end{lem}

\begin{proof}
Let $x$, $y$, $B_x = B(x, r_x )$ and $B_y = B(y, r_y )$. We can
assume that $\displaystyle |x-y| \leq \frac{r_x}{2}$ since otherwise
the conclusion is trivial. A simple computation gives
\begin{align*}
& Mu(x) - Mu(y) =   \fint_{B_x}\hspace{-0.2cm}u\, d\mu - \fint_{B_y}\hspace{-0.2cm}u\,  d\mu   = \\
&    \frac{1}{\mu (B_x )}\int_{B_x \setminus B_y} \hspace{-0.5cm} u
\, d\mu - \frac{1}{\mu (B_y )}\int_{B_y \setminus B_x}
\hspace{-0.5cm}u \, d\mu + \frac{\mu (B_y ) - \mu (B_x )}{\mu (B_x)
\mu (B_y)} \int_{B_x \cap B_y}\hspace{-0.6cm}u \, d\mu
\end{align*}
In particular
\begin{align*}
 & |Mu(x) - Mu(y)|  \leq \\
& ||u||_{\infty} \Big [ \frac{\mu (B_x \setminus B_y ) }{\mu (B_x )}
+ \frac{\mu (B_y \setminus B_x )}{\mu (B_y )} + \frac{\mu (B_x
\setminus B_y ) + \mu (B_y \setminus B_x )
}{max\{\mu (B_x ) , \mu (B_y )\}} \Big ]  \leq \\
 &   2||u||_{\infty}\Big [  \frac{\mu (B_x \setminus B_y ) }{\mu
(B_x )} + \frac{\mu (B_y \setminus B_x )}{\mu (B_y )} \Big ]
\end{align*}
so $(2.5)$ follows from Lemma $2.2$.
\end{proof}

We recall now that a  \textit{concave modulus of continuity} is a
non-decreasing concave function $\omega : [0, +\infty ) \to [0,
+\infty )$ such that $\omega (0) = 0$. If $\Omega \subset \R^d $ is
convex and $u\in C(\Omega )$ we will denote by $\displaystyle
\omega_{u, \Omega}$ the (lowest) concave modulus of continuity of
$u$ in $\Omega$ so, in particular
$$
|u(x) - u(y)| \leq \omega_{u, \Omega}(|x-y|)
$$
for $x$, $y\in \Omega$.

Consider the operators $S$, $M$ and $T_{\alpha}$ given by $(1.2)$,
$(1.3)$ and $(1.4)$ respectively, where $0\leq \alpha < 1$. Suppose
now that $\Omega\subset \R^d$ is convex, $r$ is an admissible radius
function in $\Omega$ satisfying the lipschitz condition $(1.1)$,
$u\in C(\overline{\Omega})$ and $G$ is a proper convex subdomain of
$\Omega$. If $x$,$y\in G$ then the proof of Proposition $3.2$ in
\cite{LGA} actually shows that
\begin{equation}
|Su(x) - Su(y)| \leq \omega_{u, \widetilde{G}} ( |x-y| )
\end{equation}
where $\widetilde{G}$ is the convex hull of $\displaystyle
\bigcup_{x\in G} B_x $ and
 $\omega_{u, \widetilde{G}}$ stands for the (concave) modulus of continuity of
$u$ in $\widetilde{G}$ (see \cite{LGA}).

\begin{prop}
Let $\mu$ be a doubling measure on $\R^d$, $\Omega \subset \R^d$ a
bounded, convex domain domain, $r: \Omega \to (0, +\infty )  $ an
admissible radius function in $\Omega$ satisfying $(1.1)$, $u\in
C(\overline{\Omega} )$ and $G \Subset \Omega$ a proper convex
sub-domain of $\Omega$.  There are constants $C>0$ and $ 0 < \theta
\leq 1$ only depending on $\mu$ and $d$ such that, if $x$, $y\in G$
then
\begin{equation}
|T_{\alpha}u(x) - T_{\alpha}u(y)| \leq  \alpha \omega_{u,
\widetilde{G}} (|x-y| ) + (1-\alpha ) C ||u||_{\infty} \Big (
\frac{|x-y|}{r_x} \Big )^{\theta}
\end{equation}
where
$$
\widetilde{G} = co \big ( \bigcup_{x\in G} B_x \big )
$$
In particular, if $r_x \geq t_1 >0$ for all $x\in G$ then
\begin{equation}
\omega_{T_{\alpha}u, G} (t) \leq \alpha \omega_{u, \widetilde{G}}(t)
+ (1-\alpha )C||u||_{\infty}t_1^{-\theta} t^{\theta}
\end{equation}
for $0 \leq t \leq diam (G)$.
\end{prop}

\begin{proof}
Combine $(2.6)$ and Lemma $2.3$.
\end{proof}

The next proposition justifies the choice of the normalization
constant $\lambda$ in $(1.12)$ and $(1.13)$.It implies in particular
that (admissible) radius functions of the form $r(x) = \lambda \,
dist(x,\partial \Omega )^{\beta}$ satisfy the lipschitz condition
$(1.1)$.

\begin{prop}
Let $\Omega \subset \R^d$ be a bounded, convex domain and let $\beta
\geq 1$. Then for every $x$, $y\in \Omega$, the following inequality
holds
$$
|dist(x, \partial \Omega )^{\beta } - dist(y, \partial \Omega
)^{\beta} | \leq \beta \Big ( \frac{diam (\Omega) }{2}\Big)^{\beta
-1} |x-y|
$$
In particular, the function $ \lambda_{\Omega , \beta} \, dist (
\cdot
 , \partial \Omega )^{\beta}$
is lipschitz with constant $1$ in $\Omega$, where

\begin{equation}
\lambda_{\Omega , \beta} = \frac{1}{\beta} \Big ( \frac{2}{diam
(\Omega ) } \Big ) ^{\beta -1}
\end{equation}
\end{prop}
\begin{proof}
Use the mean value theorem applied to the function $t\to t^{\beta}$
together with the fact that $\displaystyle dist(x, \partial \Omega )
\leq \frac{1}{2}diam (\Omega ) $ for each $x\in \Omega$.
\end{proof}

Let $\Omega \subset \R^d$ a bounded, convex domain and $0< \epsilon
< 1$. For $n \in \N$ define
\begin{equation}
\Omega_n = \{ x\in \Omega : \, dist(x, \partial \Omega ) >
(1-\epsilon )^n \}
\end{equation}
Observe that there is $n_0 = n_0 (\epsilon , \Omega ) $ such that
$diam (\Omega _n ) \geq \frac{1}{2}diam (\Omega )$ if $n \geq n_0$.

\begin{prop}
Let $0\leq \alpha < 1$, $0 < \epsilon < 1 -\alpha$ and suppose that
$\beta \geq 1$ and $\lambda >0$ satisfy $(1.12)$.
Let $\mu$ be a doubling measure on $\R^d$, $\Omega \subset \R^d$  a
bounded, convex domain and $r$ an admissible radius function in
$\Omega$ satisfying $(1.1)$ and $(1.13)$ for all $x\in \Omega$.
Then there are constants $C>0$ and $0< \theta \leq 1$ depending only
on $\mu$ and $d$ such that for any $u\in C(\overline{\Omega )}$ and
each $0 \leq t \leq \frac{1}{2}diam(\Omega )$ we have
\begin{equation}
\omega_{T_{\alpha} u , \Omega_n}(t) \leq \alpha \omega_{u,
\Omega_{n+1}}(t) + (1-\alpha
)C||u||_{\infty}\lambda^{-\theta}(1-\epsilon )^{-n\beta \theta} \,
t^{\theta}
\end{equation}
where $n\geq n_0$ and $\Omega_n$ is as in $(2.10)$.
\end{prop}

\begin{proof}
$(2.11)$ is consequence of $(2.8)$, condition $(1.13)$ and
Proposition $2.1$ applied to the convex subdomain $\Omega_n$ with
the choice $\displaystyle t_1 = \lambda (1-\epsilon )^{n\beta}$.

\end{proof}

We now iterate $(2.11)$.

\begin{prop}

Let $0\leq \alpha < 1$, $0 < \epsilon < 1 -\alpha$ and $\beta$,
$\lambda$ satisfying $(1.11)$ and $(1.12)$.
Let $\mu$ be a doubling measure on $\R^d$, $\Omega \subset \R^d$  a
bounded, convex domain and $r$ an admissible radius function in
$\Omega$ satisfying $(1.1)$ and $(1.13)$ for all $x\in \Omega$. Let
$G\Subset \Omega$ be a convex domain with $diam(G) \geq
\frac{1}{2}diam(\Omega )$. Then there are constants $0< \theta \leq
1$, only depending on $\mu$ and $d$ and $A>0$, depending on
$\Omega$, $\alpha$, $\beta$, $\epsilon$, $\lambda$, $\mu$, $d$ and
$G$ such that for any $u\in C(\overline{\Omega})$, and any $k\geq 1
$ we have

\begin{equation}
 \omega_{T^k_{\alpha}u, G } (t) \leq \alpha^k \omega_{u,
 \Omega}(t) + A ||u||_{\infty} t^{\theta}
\end{equation}
for $0 \leq t \leq \frac{1}{2}diam(\Omega )$, where $T^k_{\alpha} u$
stands for the $k$-th. iterate of the operator $T_{\alpha}$.
\end{prop}

\begin{proof}

Choose $n\geq n_0$ such that  $ G \subset \Omega_n$, where
$\Omega_n$ is given by $(2.10)$. We will show that
\begin{equation}
 \omega_{T^k_{\alpha}u, \Omega_n } (t) \leq \alpha^k \omega_{u,
 \Omega}(t) + (1-\alpha )\frac{C ||u||_{\infty}(1- \epsilon )^{-n\beta \theta}\lambda^{-\theta}}
 {1 - \alpha (1- \epsilon ) ^{-\beta}} \,  t^{\theta}
\end{equation}
Given $G$ as in the statement of the proposition, fix $n$ so that $G
\subset \Omega_n $. Then $(2.12)$ follows from $(2.13)$ by choosing
$$
 A = (1-\alpha )\frac{C (1- \epsilon )^{-n\beta
\theta}\lambda^{-\theta}}
 {1 - \alpha (1- \epsilon ) ^{-\beta }}
$$
To prove $(2.13)$ we iterate  $(2.11)$ to obtain

\begin{equation}
\omega_{T_{\alpha}^k u , \Omega_n} (t) \leq \alpha^k \omega_{u,
\Omega_{n+k}}(t) + B ||u||_{\infty}\, t^{\theta}
\sum_{j=0}^{k-1}\alpha^j (1-\epsilon )^{-\beta \theta j}
\end{equation}
where $\displaystyle B = (1-\alpha ) C \lambda^{-\theta} (1-\epsilon
)^{-n\beta \theta}$. Then, $(2.13)$ follows from $(2.14)$, $(1.11)$
and the fact that $\omega_{u, \Omega_{n+k}}(\cdot ) \leq \omega_{u,
\Omega}(\cdot )$.
\end{proof}

The following proposition is the analogous of Proposition $2.4$ for
the operator $H_{\alpha}$ given by $(1.10)$.

\begin{prop}
Let $\alpha$, $\beta$, $\epsilon$, $\lambda$, $\mu$, $\theta$,
$\Omega$, $r$, and $G$ be as in Proposition $2.4$. Then there are
constants $0< \theta \leq 1$, only depending on $\mu$ and $d$ and
$A>0$, depending on $\Omega$, $\alpha$, $\beta$, $\epsilon$,
$\lambda$, $\mu$, $d$ and $G$ such that for any $u\in
C(\overline{\Omega})$, and any $k\geq 1 $ we have

\begin{equation}
 \omega_{H^k_{\alpha}u, G } (t) \leq \big( \frac{1+\alpha}{2} \big )^k \omega_{u,
 \Omega}(t) + A ||u||_{\infty} t^{\theta}
\end{equation}
for $0 \leq t \leq \frac{1}{2}diam(\Omega )$, where $H^k_{\alpha} u$
stands for the $k$-th. iterate of the operator $H_{\alpha}$.

\end{prop}

\begin{proof}
Since
$$
H_{\alpha}u(x) - H_{\alpha}u(y) = \frac{1}{2}(u(x) - u(y)) +
\frac{1}{2}( T_{\alpha}u(x) - T_{\alpha}u(y) )
$$
it follows that the analogue of $(2.11)$ for the operator
$H_{\alpha}$ reads
\begin{align}
\omega_{H_{\alpha}u, \Omega_n}(t) & \leq \frac{1}{2}\omega_{u,
\Omega_n}(t) + \frac{\alpha}{2}\omega_{u, \Omega_{n+1}}(t) \nonumber
\\ &  + \frac{1-\alpha}{2}C||u||_{\infty}\lambda^{-\theta}(1-\epsilon )^{-n\beta \theta} t^{\theta}
\end{align}
Now we iterate $(2.16)$ to obtain

\begin{align}
\omega_{H^k_{\alpha}u, \Omega_n}(t) & \leq 2^{-k}\sum_{j=0}^k {k
\choose j}\alpha^j \omega_{u, \Omega_{n+j}}(t) \nonumber \\  & + B'
||u||_{\infty}\, t^{\theta}\sum_{j=0}^{k-1} 2^{-j}\big [ 1 + \alpha
(1-\epsilon )^{-\beta \theta} \big]^{j}
\end{align}
where $\displaystyle B' = \frac{1-\alpha}{2}C \lambda^{-\theta}
(1-\epsilon )^{-n\beta \theta}$ so, from $(2.17)$ and $(1.11)$ we
get
\begin{equation}
\omega_{H^k_{\alpha}u, \Omega_n}(t) \leq \big (\frac{1 +
\alpha}{2}\big)^{k} \omega_{u, \Omega}(t) + (1-\alpha )
\frac{C||u||_{\infty} (1-\epsilon )^{-n\beta \theta}
\lambda^{-\theta}}{1 - \alpha (1-\epsilon )^{-\beta} }\, t^{\theta}
\end{equation}
and $(2.15)$ follows from $(2.18)$ as in Proposition $2.4$.
\end{proof}

For a fixed $u \in C(\overline{\Omega})$ we deduce as a consequence
the equicontinuity of the sequences $\{ T^k_{\alpha}u \}$ and $\{
H^k_{\alpha}u \}$.

\begin{prop}
Let $\alpha$, $\beta$, $\epsilon$, $\lambda$,$\mu$, $\Omega$ and $r$
be as in Proposition $2.4$. Then for any $u \in
C(\overline{\Omega})$ and each $x\in \Omega$, the sequences
$\displaystyle \{ T^k_{\alpha}u \}_{k}$ and $\displaystyle \{
H^k_{\alpha}u \}_{k}$ are equicontinuous at $x$.
\end{prop}

\begin{proof}
Choose a proper subdomain $G \Subset \Omega$ containing $x$ and
apply Propositions $2.4$ and $2.5$.

\end{proof}

\section{Equicontinuity at the boundary}

In this section we assume that $\mu$ is Lebesgue measure on $\R^d$.
Let $\Omega \subset \R^d$ be a bounded, convex domain. For $u\in
C(\overline{\Omega})$ let

$$
 G_u = \{ (x, u(x)): x\in \overline{\Omega} \}\subset \R^{d+1}
$$
be the graph of $u$ and define $\Gamma_u = co ( G_u )$ to be the
convex hull of $G_u$.

\begin{prop}
Let $\mu = m $ be Lebesgue measure on $\R^d$, $\Omega \subset \R^d$
a bounded, convex domain, $r$ an admissible radius function in
$\Omega$, $0\leq \alpha < 1$ and $S$, $M$, $T_\alpha $ and
$H_{\alpha}$ the operators given by $(1.2)$, $(1.3)$, $(1.4)$ and
$(1.11)$ respectively. If $u\in C(\overline{\Omega})$ then
\begin{equation}
G_{Su} \cup G_{Mu} \cup G_{T_{\alpha}u} \cup G_{H_{\alpha}u} \subset
\Gamma_{u}
\end{equation}
In particular, for each $k\in \N$,
\begin{align}
G_{T^k_{\alpha}u} & \subset \Gamma_{T^k_{\alpha}u} \subset
\Gamma_{T^{k-1}_{\alpha} u} \subset \cdots \subset \Gamma_u \\
G_{H^k_{\alpha}u} & \subset \Gamma_{H^k_{\alpha}u} \subset
\Gamma_{H^{k-1}_{\alpha} u} \subset \cdots \subset \Gamma_u
\end{align}
\end{prop}

\begin{proof}

Since
\begin{align*}
(x, T_{\alpha}u(x)) = & \, \alpha (x, Su(x)) + (1-\alpha )(x, Mu(x)) \\
(x, H_{\alpha}u(x)) = & \, \frac{1}{2}(x, u(x)) + \frac{1}{2}(x,
T_{\alpha}u(x))
\end{align*}
and $\Gamma_u$ is convex, to prove $(3.1)$ it is enough to prove
that $G_{Su} \subset \Gamma_u $ and $G_{Mu} \subset \Gamma_u$. Fix
$x\in \Omega$.


Let us first show that $G_{Su} \subset \Gamma_u$. It is enough to
show that there is $h\in \R^n$, with $|h|\leq r_x$ so that
\begin{equation}
\sup_{B_x} u + \inf_{B_x} u = u(x+h) + u(x-h)
\end{equation}
Indeed, if $(3.4)$ is true then
$$
(x, Su(x)) = \frac{1}{2}(x+h, u(x+h)) + \frac{1}{2}(x-h, u(x-h)) \in
\Gamma_u
$$
We may assume that $\displaystyle \sup_{B_x} u = 1$ and
$\displaystyle \inf_{B_x} u = -1$ (otherwise replace $u$ by
$\displaystyle 1 + \frac{2}{M-m}(u-M)$, where $\displaystyle
\sup_{B_x} u = M $ and $\displaystyle \inf_{B_x} u = m$). Then we
must show that there is $h\in \overline{B}(0, r_x )$ so that $u(x
+h) + u(x-h) = 0$. Define the continuous function $v$ in
$\overline{B}(0, r_x )$ as
$$
v(y) = u(x+y) + u(x-y)
$$
and choose $h_+$, $h_- \in \overline{B}(0, r_x )$ such that $u(x+h_+
) = 1$, $u(x + h_- ) = -1$. Then
\begin{align*}
v(h_+ ) = u(x + h_+ ) + u(x- h_+ ) = & 1 + u(x - h_+ ) \geq 0 \\
v(h_- ) = u(x + h_- ) + u(x - h_- ) = & u(x + h_- ) -1 \leq 0
\end{align*}
so by continuity there must be $h \in \overline{B}(0, r_x )$ such
that $v(h) = 0$. This proves $(3.4)$ and therefore the inclusion
$G_{Su} \subset \Gamma_u$.

We prove now that $G_{Mu} \subset \Gamma_u$. Observe that
\begin{equation}
\int_{B_x} u \, dm = \frac{1}{2}\int_{B(0, r_x )} [u(x+y) +
u(x-y)]dm(y)
\end{equation}
which implies, by continuity, that there is  $h\in B(0, r_x )$ such
that
$$
\fint_{B_x} u \, dm = \frac{1}{2}[u(x+h) + u(x-h)]
$$
Then
$$
(x, Mu(x)) = \frac{1}{2}\big ( (x+h, u(x+h)) + (x-h, u(x-h)) \big )
$$
and this shows that  $G_{Mu} \subset \Gamma_u$.

Now, from $(3.1)$ we have
$$
 \Gamma_{T^k_{\alpha}u} =
co(G_{T^k_{\alpha}u} ) \subset co (\Gamma_{T^{k-1}_{\alpha}u} ) =
\Gamma_{T^{k-1}_{\alpha}u}
$$
which implies $(3.2)$. The argument for $(3.3)$ is analogous.

\end{proof}

\begin{remark}
The fact that $\mu = m$ is Lebesgue measure has been used in
identity $(3.5)$.
\end{remark}

\begin{lem}
Let $\Omega \subset \R^d$ be a bounded, strictly convex domain, $u
\in C(\overline{\Omega})$ and $\Gamma_u = co (G_u )$. Then, for each
$\xi \in \partial \Omega$
\begin{equation}
\Gamma_{u} \cap ( \{ \xi \} \times \R ) = \{ (\xi , u (\xi )) \}
\end{equation}
\end{lem}

\begin{proof}
If $(\xi , t) \in \Gamma_u $ then there are $\lambda_1 , \cdots ,
\lambda_m \geq 0$, with $\displaystyle \sum_1^m \lambda _i = 1$ and
there exist $x_1 ,\cdots , x_m \in \overline{\Omega}$ such that
\begin{equation}
(\xi , t) = \sum_{i=1}^m \lambda_i ( x_i , u(x_i ))
\end{equation}
From the strict convexity we deduce that the convex combination in
$(3.6)$ must be trivial in the sense that, say,  $\lambda_1 = 1$,
$\lambda_2 = \cdots = \lambda_m =0$. Then $x_1 = \xi$,  $t= u(\xi )$
and $(3.5)$ follows.
\end{proof}

\begin{prop}
Let $ \mu = m $ be Lebesgue measure on $\R^d$, $\Omega \subset \R^d$
a bounded, strictly convex domain, $r$ an admissible radius function
in $\Omega$, $0\leq \alpha < 1$ and $S$, $M$, $T_\alpha$ and
$H_{\alpha}$ the operators given by $(1.2)$, $(1.3)$, $(1.4)$ and
$(1.11)$ respectively. Then for any $u\in C(\overline{\Omega})$ and
each $\xi \in \partial \Omega$, the sequences $\displaystyle
\{T^k_{\alpha}u \}_k$ and $\displaystyle \{ H^k_{\alpha}u \}_k$ are
equicontinuous at $\xi $.
\end{prop}

\begin{proof}
Fix $u\in C(\overline{\Omega} )$. Let $\xi \in \partial \Omega$ and
suppose that $\{ T^k_{\alpha}u\}$ is not equicontinuous at $\xi$.
Then there are $\epsilon >0$ and sequences $\{ k_j \} \subset \N $,
$\{ x_j \} \subset \overline{\Omega}$  with  $k_j \uparrow \infty$,
$x_j \to \xi$ such that
$$
|T^{k_j}_{\alpha} u(x_j ) - u(\xi ) | \geq \epsilon
$$
We can assume (otherwise we could take a further subsequence) that
$\displaystyle T^k_{\alpha}u(x_j ) \to t\in \R $ and that $|t -
u(\xi ) | \geq \frac{\epsilon}{2}$. By Proposition $3.1$,
$\displaystyle (x_j , T^k_{\alpha}u(x_j ) ) \in \Gamma_u$ which is a
closed set, so $(\xi , t) \in \Gamma_u$. The contradiction then
follows from Lemma $3.1$.  Therefore $\displaystyle \{ T^k_{\alpha}u
\}_k$ is equicontinuous at each point of $\partial \Omega$. The same
argument  provides equicontinuity of $\displaystyle \{ H^k_{\alpha}
u \}_k $ at each point of $\partial \Omega$.
\end{proof}

\section{Proof of Theorem 1}

The uniqueness part follows from the next comparison principle,
which holds under much more general assumptions.

\begin{prop}
Let $\mu$ be a positive Borel measure in $\R^d$ with the property
that $\mu (B) >0$ for every ball $B \subset \R^d$. Let $\Omega
\subset \R^d$ be a bounded domain, $r: \Omega \to (0, +\infty )$ an
admissible one-radius function in $\Omega$, $0\leq \alpha < 1$ and
let $S$, $M$, $T_{\alpha}$ be the operators given by $(1.2)$,
$(1.3)$ and $(1.4)$ respectively. Suppose that $u$ and $v\in
C(\overline{\Omega})$ satisfy $T_\alpha u = u$, $T_\alpha v = v$ and
that $u\leq v$ on $\partial \Omega $. Then $u\leq v$ in $\Omega$.
\end{prop}

\begin{proof}
The argument is standard in comparison results. Let $u$ and $v$ as
in the statement of the proposition. Let $\displaystyle m =
\max_{\overline{\Omega}}(u - v) $. We will show that $m \leq 0$.
Suppose, on the contrary, that $m>0$ and define
$$
A = \{ x\in \Omega \, : \, (u-v)(x) = m \}
$$
Then $A$ is a nonempty, closed subset of $\Omega$. Take $a \in A$.
We will see that $B_a \subset A$ so $A$ is also open. Indeed,
$$
u(a) = \alpha Su(a) + (1-\alpha ) Mu(a) = \alpha ( m + Sv(a)) +
(1-\alpha ) ( m + Mv(a))
$$
Since $u \leq v + m$ in $B_a$ we must have in particular that $Mu(a)
= m + Mv(a)$. Therefore
\begin{equation}
\int_{B_a} (v+m-u)\, d\mu = 0
\end{equation}
The integrand in $(4.1)$ is  continuous  and nonnegative so by
continuity and the hypothesis on $\mu$ it follows that $v + m - u
\equiv  0$ in $B_a$. This proves that $A$ is open. Therefore, by
connectedness $A = \Omega$ and $v-u \equiv m > 0$ in $\Omega$, which
contradicts the assumption $u\leq v$ on $\partial \Omega$. Then
$m\leq 0$ and the proposition follows.
\end{proof}

To prove the existence part we will need a result from metric fixed
point theory. Let $(X , || . ||)$ be a Banach space and $K \subset
X$. A self-mapping $T: K \to K $ of $K$ is \emph{nonexpansive} if
$$
||Tx - Ty || \leq ||x -y ||
$$
for each $x$, $y\in K$. The following result will be a key
ingredient in the proof of existence. It is a particular case of a
more general result from Ishikawa (\cite{I} , see also \cite{GK},
Theorem $9.4$).

\begin{theorem} \rm{(Ishikawa)}
Let $X$ be a Banach space, $K\subset X$ a bounded, closed and convex
subset of $X$ and let $T: K \to K$ be a nonexpansive self-mapping of
$K$. Define $\displaystyle H = \frac{1}{2}(I + T)$. Then
\begin{equation}
\lim_{k\to \infty}||H^{k+1}x - H^k x || = 0
\end{equation}
for each $x\in K$.
\end{theorem}

\

Condition  $(4.2)$ has been named \textit{asymptotic regularity} by
some authors. Let us see how to prove the existence in Theorem $1$.
Take $X = ( C(\overline{\Omega}) , || . ||_{\infty} ) $ and fix
$f\in C(\partial \Omega )$. Define
$$
K = \{ u\in X : \, u|_{\partial \Omega} = f \, , \, ||u ||_{\infty}
= ||f ||_{\infty} \}
$$
Then $K$ is a nonempty bounded, closed, convex subset of $X$.
Observe that $S$, $M$, $T_{\alpha}$ and $H_{\alpha}$ are all
nonexpansive self-mappings of $K$. To prove the existence of the
solution of the Dirichlet problem $(1.5)$ it is enough to show that
$T_{\alpha}$ has a fixed point in $K$.

Choose $u_0 \in K$. Then the sequence $\{ H^k_{\alpha} u_0 \}$ is
pointwise bounded and also equicontinuous at each point of the
compact set $\overline{\Omega}$ by Propositions $2.6$ and $3.2$.
Therefore, by Arzéla-Ascoli theorem there are a subsequence $\{ k_j
\} $, with $k_j \uparrow \infty$ and $\widetilde{u}\in K$ such that
\begin{equation}
\lim_{j\to \infty} H^{k_j}_{\alpha} u_0 = \widetilde{u}
\end{equation}
uniformly in $\overline{\Omega}$. Then
\begin{equation}
\lim_{j\to \infty}H^{k_j +1}_{\alpha}u_0 =  H_{\alpha}\widetilde{u}
\end{equation}
and, from  Theorem $4.1$ applied to $H_{\alpha}$, we get
\begin{equation}
\lim_{j\to \infty} || H^{k_j +1}_{\alpha}u_0 - H^{k_j}_{\alpha}u_0
||_{\infty} = 0
\end{equation}
so, from $(4.3)$, $(4.4)$ and $(4.5)$ we deduce that
$H_{\alpha}\widetilde{u} = \widetilde{u}$. Since $T_{\alpha}$ and
$H_{\alpha}$ have the same fixed points, we get $T_{\alpha}
\widetilde{u} = \widetilde{u}$ which proves the existence part in
Theorem $1$. To see that, actually, $\displaystyle H^k_{\alpha}u_0
\to \widetilde{u}$ suppose, on the contrary, that there are
$\epsilon
>0$ and a subsequence $\{ m_j \}$ such that
\begin{equation}
||H^{m_j}_{\alpha}u_0 - \widetilde{u}||_{\infty} \geq \epsilon
\end{equation}
By equicontinuity, a subsequence of $\displaystyle
\{H^{m_j}_{\alpha}u_0 \}$ would converge to some $v\in K$ and, as
before, $T_{\alpha}v = v$ so, by uniqueness, $v = \widetilde{u}$,
which contradicts $(4.6)$. This proves that $\displaystyle
H^k_{\alpha}u_0 \to \widetilde{u}$ and finishes the proof of Theorem
$1$.

\end{document}